\documentclass[10pt]{amsart}
\usepackage{indentfirst,enumerate,cite,amssymb,amsfonts,amsmath,amsthm,mathrsfs,dsfont}
\usepackage{color}
\usepackage{multicol}
\usepackage[colorlinks=true,linkcolor=blue,citecolor=blue]{hyperref}
\theoremstyle{plain}
\newtheorem{thm}{Theorem}[section]
\newtheorem{prop}[thm]{Proposition}

\newtheorem{lemma}[thm]{Lemma}
\newtheorem{defi}[thm]{Definition}
\newtheorem{rem}[thm]{Remark}

\numberwithin{equation}{section}
\newcommand{\HOM}{\operatorname{Hom}}
\newcommand{\KER}{\operatorname{Ker}}
\newcommand{\Hom}[3]{\HOM_{#1}(#2,#3)}

\newcommand{\Rep}{\mbox{Rep}}
\newcommand{\HS}{{\mathtt{HS}}}
\newcommand{\Tr}{\mbox{\emph{Tr}}}
\usepackage{geometry}
\newcommand{\N}{{\mathbb{N}}}
\newcommand{\Q}{\mathbb{Q}}

\newcommand{\Z}{\mathbb{Z}}
\newcommand{\C}{\mathbb{C}}

\newcommand{\St}{{\mathbb S}^3}
\newcommand{\Sth}{\widehat{{\mathbb S}^3}}

\newcommand{\DG}{\mathcal{D}'(G)}
\newcommand{\Ctl}{{\C^{(2\ell+1)\times (2\ell+1)}}}

\newcommand{\TS}{\mathbb{T}^1\times\St}

\newcommand{\p}[1]{{\left({#1}\right)}}
\newcommand{\jp}[1]{{\left\langle{#1}\right\rangle}}

\newlength{\dhatheight}
\newcommand{\doublehat}[1]{%
	\settoheight{\dhatheight}{\ensuremath{\hat{#1}}}
	\addtolength{\dhatheight}{-0.15ex}
	\widehat{\vphantom{\rule{5pt}{\dhatheight}}%
		\smash{\widehat{#1}}}}

\begin{document}
	
	\title[Partial Fourier Series on Compact Lie Groups]
	{Partial Fourier Series on Compact Lie Groups}
	

\author[Alexandre Kirilov]{Alexandre Kirilov}
\address{
	Universidade Federal do Paran\'{a}, 
	Departamento de Matem\'{a}tica,
	C.P.19096, CEP 81531-990, Curitiba, Brazil
}
\email{akirilov@ufpr.br}


\author[Wagner de Moraes]{Wagner A. A. de Moraes}
\address{
	Universidade Federal do Paran\'{a},
	Programa de P\'os-Gradua\c c\~ao de Matem\'{a}tica,
	C.P.19096, CEP 81531-990, Curitiba, Brazil
}
\email{wagneramat@gmail.com}


\author[Michael Ruzhansky]{Michael Ruzhansky}
\address{Ghent University, 
	Department of Mathematics: Analysis, Logic and Discrete Mathematics, 
	Ghent, Belgium 
	and 
	Queen Mary University of London, 
	School of Mathematical Sciences, 
	London, United Kingdom
}
\email{Michael.Ruzhansky@ugent.be}
	
	\date{\today}
	
	\subjclass{Primary 22E30, 43A77 ; Secondary 58D25, 43A25 }
	\keywords{compact Lie group, partial Fourier series, evolution equation}
	
	\begin{abstract}
	In this note we investigate the partial Fourier series on a product of two compact Lie groups. We give necessary and sufficient conditions for a sequence of partial Fourier coefficients to define a smooth function or a distribution. 
	As applications, we will study conditions for the global solvability of an evolution equation defined on $\TS$ and we will show that some properties of this evolution equation can be obtained from a constant coefficient equation.
	\end{abstract}

	\maketitle	
\tableofcontents
\raggedbottom
	\section{Introduction}

The study of global properties for vector fields and systems of vector fields defined on closed manifolds has made a great advance in the last decades, especially in the case where the manifold is a torus or a product of a torus by another closed manifold. See, for example, the impressive list of articles \cite{BCM93,BCP04,CC00,GPY92,GPY93,GW72,GW73a,GW73b,Hou82,Pet11,Pet98} and references therein, discussing problems related to the global solvability and global hypoellipticity in the smooth, analytical and ultradifferentiable senses.

One of the main tools in these references is the characterization of the functional environments via the asymptotic behavior of its Fourier coefficients, especially of the partial Fourier coefficients with respect to one or more variables.

For example, in the case of evolution equations, the use of partial Fourier series in the spatial variable can reduce the problem to find solutions to a sequence of ordinary differential equations depending on the time, for which the methods of control of asymptotic behavior are much more developed and understood. This technique is used in 
\cite{AGK18,AGKM18,AK19,BDGK15,BK07-JFA,Ber99,Hou79} and several others references.

A natural step in the development of the theory is to extend such techniques to the study of global properties of vector fields in compact Lie Groups. For instance, consider the operator $P: C^\infty (\mathbb{T}^1 \times \St) \to C^\infty(\mathbb{T}^1\times \St)$   defined by
\begin{equation}\label{1}
Pu(t,x) := \partial_tu(t,x) + a(t)Xu(t,x),
\end{equation}
where $X$ is a left-invariant vector field on $\St$ and $a \in C^\infty(\mathbb{T}^1)$ is a real-valued function. 

Taking the Fourier series with respect to the second variable, we obtain a system of ordinary differential equations on $t$, which can be easily solved. The analysis of the behavior of these solutions at infinity will allow us to say if the corresponding Fourier partial series converges to a smooth (or distributional) solution of the problem. 

In this work we define precisely the partial Fourier coefficients of functions and distributions, and we obtain the characterization of functions and distributions by its  partial Fourier series on a product of compact Lie groups.

The paper is organized as follows.	In Section 2 we present some classical results about Fourier analysis on compact Lie groups and fixed the notation that will be used henceforth. In Section 3 first we establish the relation of the Fourier analysis of a product of two compact Lie groups with the Fourier analysis of each compact Lie group. Then we define the partial Fourier coefficients and we give necessary and sufficient conditions for a sequence of partial Fourier coefficients to define a smooth function or a distribution. 
In section 4 we present two applications of the theory developed in the previous section. First we will study the global $C^\infty$-solvability of the operator \eqref{1} through its partial Fourier coefficients. Then, we present a conjugation that transform the operator \eqref{1} in a constant coefficient operator. 
\raggedbottom
	\section{Fourier analysis on compact Lie groups}
	
In this section we introduce most of the notations and preliminary results necessary for the development of this work. A very careful presentation of these concepts and the demonstration of all the results presented here can be found in the references \cite{FR16}  and \cite{livropseudo}.

Let $G$ be a compact Lie group and let $\Rep(G)$  be the set of continuous irreducible unitary representations of $G$. Since $G$ is compact, every continuous irreducible unitary representation $\phi$ is finite dimensional and it can be viewed as a matrix-valued function $\phi: G \to \C^{d_\phi\times d_\phi}$, where $d_\phi = \dim \phi$. We say that $\phi \sim \psi$ if there exists an unitary matrix $A\in C^{d_\phi \times d_\phi}$ such that $A\phi(x) =\psi(x)A$, for all $x\in G$. We will denote by $\widehat{G}$ the quotient of $\Rep(G)$ by this equivalence relation.
	
	For $f \in L^1(G)$ we define the group Fourier transform of $f$ at $\phi$ as
	\begin{equation}\label{deffourier}
	\widehat{f}(\phi)=\int_G f(x) \phi(x)^* \, dx,
	\end{equation}
	where $dx$ is the normalized Haar measure on $G$.
	By the Peter-Weyl theorem, we have that 
	\begin{equation}\label{ortho}
	\mathcal{B} := \left\{\sqrt{\dim \phi} \, \phi_{ij} \,; \ \phi=(\phi_{ij})_{i,j=1}^{d_\phi}, [\phi] \in \widehat{G} \right\},
\end{equation}
	is an orthonormal basis for $L^2(G)$, where we pick only one matrix unitary representation in each class of equivalence, and we may write
	\begin{equation*} \label{peterweyl}
	f(x)=\sum_{[\phi]\in \widehat{G}}d_\phi \emph{\Tr}(\phi(x)\widehat{f}(\phi)).
	\end{equation*}
	Moreover, the Plancherel formula holds:
	\begin{equation}
	\label{plancherel} \|f\|_{L^{2}(G)}=\left(\sum_{[\phi] \in \widehat{G}}  d_\phi \ 
	\|\widehat{f}(\phi)\|_{\HS}^{2}\right)^{1/2}=:
	\|\widehat{f}\|_{\ell^{2}(\widehat{G})},
	\end{equation}
	where 
	\begin{equation*}\label{normhs} \|\widehat{f}(\phi)\|_{\HS}^{2}=\emph{\Tr}(\widehat{f}(\phi)\widehat{f}(\phi)^{*})=\sum_{i,j=1}^{d_\phi}  \bigr|\widehat{f}(\phi)_{ij}\bigr|^2.
	\end{equation*}
	Let $\mathcal{L}_G$ be the Laplace-Beltrami operator of $G$. For each $[\phi] \in \widehat{G}$, its matrix elements are eigenfunctions of $\mathcal{L}_G$ correspondent to the same eigenvalue that we will denote by $-\nu_{[\phi]}$, where $\nu_{[\phi]} \geq 0$. Thus
	$$
	-\mathcal{L}_G \phi_{ij}(x) = \nu_{[\phi]}\phi_{ij}(x), \quad \textrm{for all } 1 \leq i,j \leq d_\phi,
	$$
	and we will denote by
	$$
	\jp \phi := \left(1+\nu_{[\phi]}\right)^{1/2}
	$$
	the eigenvalues of $(I-\mathcal{L}_G)^{1/2}.$ We have the following estimate for the dimension of $\phi$ (Proposition 10.3.19 of \cite{livropseudo}): there exists $C>0$ such that for all $[\xi] \in \widehat{G}$ it holds
\begin{equation}\label{dimension}
	d_\phi \leq C \jp{\phi}^{\frac{\dim G}{2}}.
\end{equation}
	
		For $x\in G$, $X\in \mathfrak{g}$ and $f\in C^\infty(G)$, define
	$$
	L_Xf(x):=\frac{d}{dt} f(x\exp(tX))\bigg|_{t=0}.
	$$

Notice that the operator $L_X$ is left-invariant. Indeed,
\begin{align*}
\pi_L(y)L_Xf(x) = L_Xf(y^{-1}x) \nonumber 
&= \frac{d}{dt} f(y^{-1}x\exp(tX))\bigg|_{t=0} \nonumber \\
&= \frac{d}{dt} \pi_L(y)f(x\exp(tX))\bigg|_{t=0}\nonumber \\
&= L_X\pi_L(y)f(x),\nonumber
\end{align*}
for all $x,y\in G$.

When there is no possibility of ambiguous meaning, we will write only $Xf$ instead of $L_Xf$. 

	Let $G$ be a compact Lie group of dimension $n$ and let $\{X_i\}_{i=1}^n$ be a basis of its Lie algebra. For a multi-index $\alpha =(\alpha_1,\cdots \alpha_n)\in \N_0^n$, we define the left-invariant differential operator of order $|\alpha|$
	\begin{equation*}\label{partial}
	\partial^\alpha := Y_1 \cdots Y_{|\alpha|},
	\end{equation*}
	with $Y_j \in \{X_i \}_{i=1}^n$, $1 \leq j \leq |\alpha|$ and $\sum\limits_{j:Y_j=X_k}1=\alpha_k$ for every $1\leq k \leq n$. It means that $\partial^\alpha$ is a composition of left-invariant derivatives with respect to vectors $X_1, \dots, X_n$ such that each $X_k$ enters $\partial^\alpha$ exactly $\alpha_k$ times. We do not specify in the notation $\partial^\alpha$ the order of vectors $X_1,\dots,X_n$, but this will not be relevant in the arguments that we will use in this work. 
	
	We equipped $C^\infty(G)$ with the usual Fréchet space topology defined by seminorms $p_\alpha(f) = \max\limits_{x \in G} |\partial^\alpha f(x)|$. Thus, the convergence on $C^\infty(G)$ is just the uniform convergence of functions and all their derivatives: $f_k \to f$ in $C^\infty(G)$ if $\partial^\alpha f_k(x) \to \partial^\alpha f(x)$, for all $x \in G$, due to the compactness of $G$.	We define the space of distributions $\mathcal{D}'(G)$ as the space of all continuous linear functional on $C^\infty(G)$.	
	For $u \in \DG$, we define the distribution $\partial^\alpha u$ as
	$$
	\jp{\partial^\alpha u, \psi} := (-1)^{|\alpha|} \jp{u,\partial^\alpha \psi},
	$$
	for all $\psi \in C^\infty(G)$.
	
	Let $\alpha \in \N_0^n$. The symbol of $\partial^\alpha$ at $x \in G$ and  $\xi \in \Rep (G)$ is
	\begin{equation}\label{defsymbol}
		\sigma_{\partial^\alpha}(x, \phi) := \phi(x)^*(\partial^\alpha \phi)(x) \in \C^{d_\phi\times d_\phi}.
	\end{equation}
	Notice that the symbol of $\partial^\alpha$ does not depends of $x\in G$ because $\partial^\alpha$ is a left-invariant operator. 
	There exists $C_0>0$ such that 
	\begin{equation}\label{estimateop}
	\|\sigma_{\partial^\alpha}(\phi)\|_{op} \leq C_0^{|\alpha|} \jp{\phi}^{|\alpha|},
	\end{equation}
	for all $\alpha \in \N_0^n$ and $[\phi] \in \widehat{G}$ (see \cite{DR14}, Proposition 3.4), where $\|\cdot\|_{op}$ stands for the usual operator norm and, from Chapter 2 of \cite{GVL96}, we have
\begin{equation}\label{symbol}
	\|\sigma_{\partial^\alpha}(\phi)\|_{op} \leq 	\|\sigma_{\partial^\alpha}(\phi)\|_{\HS} \leq \sqrt{d_\phi}	\|\sigma_{\partial^\alpha}(\phi)\|_{op}.	
\end{equation}

For every integer $M \geq \frac{\dim G}{2}$ there exists $C_M>0$ such that 
	\begin{equation}\label{estimatexi}
	\|\phi_{ij}\|_{L^\infty(G)} \leq C_M \jp{\phi}^M,
	\end{equation}
	for all $[\phi] \in \widehat{G}$, $1\leq i,j \leq d_\phi$ (see \cite{DR14}, Lemma 3.3).
	
	Smooth functions on $G$ can be characterized in terms of
	their Fourier coefficients comparing with powers of the eigenvalues of $(I-\mathcal{L}_G)^{1/2}.$ Precisely, the following statements are equivalent:
	\begin{enumerate}[(i)] 
		\item $f \in C^\infty(G)$;
		\item for each $N>0$, there exists $C_N > 0$ such that
		$$
		\|\widehat{f}(\phi)\|_{\HS} \leq C_N \jp{\phi}^{-N},
		$$
		for all $[\phi] \in \widehat{G}$;
		\item for each $N>0$, there exists $C_N > 0$ such that
		$$
		|\widehat{f}(\phi)_{ij}| \leq C_N \jp{\phi}^{-N}, 
		$$
		for all $[\phi] \in \widehat{G}, \ 1\leq i,j \leq d_\phi.$
	\end{enumerate}

By duality, we have a similar result for distributions on $G$, 	where we define $$
\widehat{u}(\phi)_{ij}:=u(\overline{\phi_{ji}}).
$$
This definition agrees with \eqref{deffourier} when the distribution comes from a $L^1(G)$ function. The following statements are equivalent:
	\begin{enumerate}[(i)]
		\item $u \in \DG$;
		\item there exist $C, \, N > 0$ such that
		$$
		\|\widehat{u}(\phi)\|_{\HS} \leq C \jp{\phi}^{N}, 
		$$
		for all $[\phi] \in \widehat{G}$;
		\item there exist $C, \, N > 0$ such that
		$$
		|\widehat{u}(\phi)_{ij}| \leq C \jp{\phi}^{N},
		$$
		for all $[\phi] \in \widehat{G}, \  1\leq i,j \leq d_\phi.$
	\end{enumerate}

\section{Partial Fourier series}\label{partialfourier}

Let $G_1$ and $G_2$ be compact Lie groups, and set $G=G_1\times G_2$.
Let  $\xi \in \Hom{}{G_1}{\mbox{{Aut}}(V_1)}$ and $\eta \in \Hom{}{G_2}{\mbox{{Aut}}(V_2)}$. The  external tensor product representation $\xi \otimes \eta$ of $G$ on $V_1\otimes V_2$ is defined by
{\small
	$$
	\begin{array}{ccclccl}
	\xi \otimes \eta: & G_1 \times G_2 &\to& \mbox{{Aut}}(V_1 \otimes V_2)&&& \\
	& (x_1,x_2) & \mapsto &	(\xi \otimes \eta)(x_1,x_2): & V_1 \otimes V_2&\to&V_1 \otimes V_2\\
	&&&&(v_1,v_2)&\mapsto& \xi(x_1)(v_1)\otimes \eta(x_2)(v_2)
	\end{array}
	$$
}

We point out that the external tensor product of unitary representation is also unitary. Moreover, if $\xi \in \Hom{}{G}{\mbox{U}(d_\xi)}$ and $\eta \in \Hom{}{G}{\mbox{U}(d_\eta)}$ are matrix unitary representations, then $\xi\otimes\eta \in \Hom{}{G}{\mbox{U}(d_\xi d_\eta)}$ is also a matrix unitary representation and 
$$
\xi\otimes\eta(x_1,x_2) = \xi(x_1) \otimes \eta(x_2) \in \C^{d_\xi d_\eta \times d_\xi d_\eta},
$$
where  $\xi(x_1) \otimes \eta(x_2)$ is the Kronecker product of these matrices, that is,
$$
\xi(x_1) \otimes \eta(x_2) = \left(\begin{array}{ccc}
\xi(x_1)_{11}\eta(x_2) & \cdots & \xi(x_1)_{1d_\xi}\eta(x_2) \\
\vdots & \ddots & \vdots \\
\xi(x_1)_{d_\xi1}\eta(x_2) & \cdots & \xi(x_1)_{d_\xi d_\xi} \eta(x_2)
\end{array}
\right).
$$

It is enough to study continuous irreducible unitary representations of $G_1$ and $G_2$ to obtain the elements of $\widehat{G}$, since for every $[\phi] \in \widehat{G}$, there exist $[\xi] \in \widehat{G_1}$ and $[\eta] \in \widehat{G_2}$ such that
	$\phi \sim \xi\otimes\eta$,
	that is, $[\phi] = [\xi\otimes\eta] \in \widehat{G}$  and $d_\phi = d_\xi \cdot d_\eta$. Moreover,  $[\xi_1 \otimes \eta_1]= [\xi_2 \otimes \eta_2]$ if and only if $[\xi_1] = [\xi_2]$ and $[\eta_1] = [\eta_2]$.
	The proof of this fact can be found on \cite{BD95} (Chapter II, Proposition 4.14). Therefore, the map $[\xi\otimes\eta] \mapsto ([\xi],[\eta])$ is a bijection from $\widehat{G}$ to $\widehat{G_1} \times \widehat{G_2}$.

It is easy to see that $\mathcal{L}_G = \mathcal{L}_{G_1} + \mathcal{L}_{G_2}$, so $\nu_{[\xi \otimes \eta]} = \nu_{[\xi]} + \nu_{[\eta]}$. Therefore we have
\begin{equation}\label{pseudonorm}
	\frac{1}{2} (\jp{\xi} + \jp{\eta}) \leq \jp{\xi \otimes \eta} \leq \jp{\xi} + \jp{\eta},
\end{equation}
	for all $[\xi] \in \widehat{G_1}$ and $[\eta] \in \widehat{G_2}$.

Let $f \in L^1(G)$ and $[\phi] \in \widehat{G}$. Let $[\xi] \in \widehat{G_1}$ and $[\eta] \in \widehat{G_2}$ such that $[\phi] = [\xi\otimes\eta]$. Notice that
\begin{align*}
\widehat{f}(\xi\otimes\eta) &= \int_G f(x) (\xi\otimes\eta)(x)^* \, dx \nonumber \\
&= \int_{G_2}\int_{G_1}f(x_1,x_2)(\xi(x_1)\otimes \eta(x_2))^* \, dx_1dx_2 \nonumber \\
&= \int_{G_2}\int_{G_1}f(x_1,x_2)\xi(x_1)^*\otimes \eta(x_2)^* \, dx_1dx_2 \nonumber.
\end{align*}

Thus $\widehat{f}(\xi\otimes\eta) \in \C^{d_\xi d_\eta \times d_\xi d_\eta}$ with elements
\begin{align*}
\widehat{f}(\xi\otimes\eta)_{ij} &=  \int_{G_2}\int_{G_1}f(x_1,x_2)(\xi(x_1)^*\otimes \eta(x_2)^*)_{ij} \, dx_1dx_2 \nonumber \\
&=\int_{G_2}\int_{G_1} f(x_1,x_2)\overline{\xi(x_1)_{nm}} \, \overline{\eta(x_2)_{sr}} \, dx_1 dx_2 \nonumber
\end{align*}
where $1 \leq m,n \leq d_\xi$, $1 \leq r,s \leq d_\eta$ are given by
\begin{equation*}
\begin{array}{ccc}
m &=& \left\lfloor \frac{i-1}{d_\eta} \right\rfloor  +1 \nonumber,  \\
n &=&  \left\lfloor \frac{j-1}{d_\eta} \right\rfloor  +1 \nonumber,
\end{array}
\qquad
\begin{array}{ccc}
r &=& i-\left\lfloor \frac{i-1}{d_\eta} \right\rfloor d_\eta \nonumber, \\
s &=&  j-\left\lfloor \frac{j-1}{d_\eta} \right\rfloor d_\eta \nonumber. 
\end{array}
\end{equation*}

Similarly for $u \in \DG$, we have
$$
\widehat{u}(\xi\otimes\eta)_{ij} = \jp{u, \overline{(\xi \otimes \eta)_{ji}}} = \jp{u, \overline{\xi_{nm} \times \eta_{sr}} },
$$
where $(\xi_{nm} \times \eta_{sr})(x_1,x_2):= \xi(x_1)_{nm}\eta(x_2)_{sr}$.
\begin{defi}
		Let $G_1$ and $G_2$ be compact Lie groups and, set $G=G_1\times G_2$. Let $f \in L^1(G)$, $\xi \in \emph{\Rep}(G_1)$, and $x_2 \in G_2$. The  $\xi$-partial Fourier coefficient of $f$ at $x_2$ is defined by 
	$$
	\widehat{f}(\xi, x_2) = \int_{G_1} f(x_1,x_2)\, \xi(x_1)^* \, dx_1 \in \C^{d_\xi \times d_\xi},
	$$
	with components
	$$
	\widehat{f}(\xi, x_2)_{mn} = \int_{G_1} f(x_1,x_2)\, \overline{\xi(x_1)_{nm}} \, dx_1, \quad 1 \leq m,n\leq d_\xi.
	$$
	Similarly, for $\eta \in \emph{\Rep}(G_2)$ and $x_1 \in G_1$, we define the  $\eta$-partial Fourier coefficient of $f$ at $x_1$ as
	$$
	\widehat{f}(x_1,\eta) = \int_{G_2} f(x_1,x_2)\, \eta(x_2)^* \, dx_2 \in \C^{d_\eta \times d_\eta},
	$$
	with components
	$$
	\widehat{f}(x_1, \eta)_{rs} = \int_{G_2} f(x_1,x_2)\, \overline{\eta(x_2)_{sr}} \, dx_1, \quad 1 \leq r,s\leq d_\eta.
	$$
\end{defi}

By the definition, the function
	$$
	\begin{array}{rccl}
	\widehat{f}(\xi, \: \cdot\: )_{mn}: & G_2 & \longrightarrow & \C \\
	& x_2 & \longmapsto & \widehat{f}(\xi, x_2)_{mn}
	\end{array}
	$$
	belongs to $L^1(G_2)$ for all $\xi \in \Rep(G_1)$, $1 \leq m,n\leq d_\xi$. Similarly, the function
$$
	\begin{array}{rccl}
		\widehat{f}(\: \cdot \: , \eta )_{rs}: & G_1 & \longrightarrow & \C \\
		& x_1 & \longmapsto & \widehat{f}(x_1,\eta)_{rs}
	\end{array}
$$
	belongs to $L^1(G_1)$ for all $\eta \in \Rep(G_2)$, $1 \leq r,s\leq d_\eta$. 

Let $\xi \in \mbox{Rep}(G_1)$ and $\eta \in \mbox{Rep}(G_2)$. Since $\widehat{{f}}(\xi,\, \cdot\, )_{mn} \in L^1(G_2)$ for all $1\leq m,n \leq d_\xi$, we can take its $\eta$-coefficient of Fourier:
$$
\doublehat{\,f\,}\!(\xi,\eta)_{mn}:= \int_{G_2} \widehat{f}(\xi, x_2)_{mn} \eta(x_2)^* \, dx_2 \in \C^{d_\eta \times d_\eta}
$$
with components
\begin{align*}
\doublehat{\,f\,}\!(\xi,\eta)_{mn_{rs}} &= \int_{G_2} \widehat{f}(\xi, x_2)_{mn} \overline{\eta(x_2)_{sr}} \, dx_2 \nonumber \\
&= \int_{G_2}\int_{G_1} f(x_1,x_2)\overline{\xi(x_1)_{nm}} \, \overline{\eta(x_2)_{sr}} \, dx_1 dx_2, \nonumber
\end{align*}
for $1 \leq r,s \leq d_\eta$. 
Similarly, since $\widehat{f}(\:\cdot\:,\eta)_{rs} \in L^1(G_1)$ for all $1\leq r,s \leq d_\eta$, we can take its $\xi$-coefficient of Fourier:
$$
\doublehat{\,f\,}\!(\xi,\eta)_{rs}:= \int_{G_1} \widehat{f}(x_1,\eta)_{rs} \xi(x_1)^* \, dx_1 \in \C^{d_\xi \times d_\xi}
$$
with components
\begin{align*}
\doublehat{\,f\,}\!(\xi,\eta)_{rs_{mn}} &= \int_{G_1} \widehat{f}(x_1, \eta)_{rs} \overline{\xi(x_1)_{nm}} \, dx_1 \nonumber \\
&= \int_{G_1}\int_{G_2} f(x_1,x_2)\overline{\xi(x_1)_{nm}} \, \overline{\eta(x_2)_{sr}} \, dx_2 dx_1, \nonumber
\end{align*}
for $1 \leq m,n \leq d_\xi$. 

Notice that
\begin{equation*}
\doublehat{\,f\,}\!(\xi,\eta)_{mn_{rs}} = \doublehat{\,f\,}\!(\xi,\eta)_{rs_{mn}} =  \widehat{f}(\xi \otimes \eta)_{ij},
\end{equation*}
with
\begin{align*}
i = d_\eta(m-1)+ r \nonumber,  \quad 
j =  d_\eta(n-1) + s \nonumber,
\end{align*}
for all $1\leq m,n \leq d_\xi$ and $1 \leq r,s \leq d_\eta$.
\begin{defi}
	Let $G_1$ and $G_2$ be compact Lie groups, and set $G=G_1\times G_2$. Let $u \in \DG$, $\xi \in \emph{\Rep}(G_1)$ and $1\leq m,n \leq d_\xi$. The $mn$-component of  the $\xi$-partial Fourier coefficient of $u$ is the linear functional defined by
	$$
	\begin{array}{rccl}
	\widehat{u}(\xi, \: \cdot\: )_{mn}: & C^\infty(G_2) & \longrightarrow & \C \\
	& \psi & \longmapsto & \jp{\widehat{u}(\xi, \: \cdot \:)_{mn},\psi} := \jp{u,\overline{\xi_{nm}}\times\psi}_G.
	\end{array}
	$$
	In a similar way, for $\eta \in \emph{\Rep}(G_2)$ and $1 \leq r,s\leq d_\eta$, we define the $rs$-component of the $\eta$-partial Fourier coefficient of $u$ as
	$$
	\begin{array}{rccl}
	\widehat{u}( \: \cdot\:,\eta )_{rs}: & C^\infty(G_1) & \longrightarrow & \C \\
	& \varphi & \longmapsto & \jp{\widehat{u}( \: \cdot \:,\eta)_{rs},\varphi} := \jp{u,\varphi\times\overline{\eta_{sr}}}_G.
	\end{array}
	$$
\end{defi}

		By definition, $ \widehat{u}(\xi, \: \cdot\: )_{mn} \in \mathcal{D}'(G_2)$ and $\widehat{u}( \: \cdot\:,\eta )_{rs} \in \mathcal{D}'(G_1)$ for all $\xi \in \Rep(G_1)$, $\eta \in \Rep(G_2)$, $1 \leq m,n\leq d_\xi$ and $1 \leq r,s\leq d_\eta$.

Let $\xi \in \mbox{Rep}(G_1)$ and $\eta \in \mbox{Rep}(G_2)$. Since $\widehat{{u}}(\xi,\, \cdot\, )_{mn} \in \mathcal{D}'(G_2)$ for all $1\leq m,n \leq d_\xi$, we can take its $\eta$-coefficient of Fourier:
$$
\doublehat{\,u\,}\!(\xi,\eta)_{mn}:= \jp{\widehat{{u}}(\xi,\, \cdot\, )_{mn},\eta^*} \in \C^{d_\eta \times d_\eta}
$$
with components
$$
\doublehat{\,u\,}\!(\xi,\eta)_{mn_{rs}} = \jp{\widehat{{u}}(\xi,\, \cdot\, )_{mn},\overline{\eta_{sr}}} = \jp{u,\overline{\xi_{nm}}\times\overline{\eta_{sr}}}_G = \jp{u,\overline{\xi_{nm}\times\eta_{sr}}}_G,
$$
for all $1\leq r,s \leq d_\eta$. Now, since $\widehat{u}( \: \cdot\:,\eta )_{rs} \in \mathcal{D}'(G_1)$ for all $1 \leq r,s\leq d_\eta$ we can take its $\xi$-coefficient of Fourier:
$$
\doublehat{\,u\,}\!(\xi,\eta)_{rs}:= \jp{\widehat{u}( \: \cdot\:,\eta )_{rs},\xi^*} \in \C^{d_\xi \times d_\xi}
$$
with components
$$
\doublehat{\,u\,}\!(\xi,\eta)_{rs_{mn}} = \jp{\widehat{u}( \: \cdot\:,\eta )_{rs},\overline{\xi_{mn}}} = \jp{u,\overline{\xi_{nm}}\times\overline{\eta_{sr}}}_G = \jp{u,\overline{\xi_{nm}\times\eta_{sr}}}_G,
$$
for all $1\leq m,n \leq d_\xi$. Notice that
\begin{equation*}
\doublehat{\,u\,}\!(\xi,\eta)_{mn_{rs}} = \doublehat{\,u\,}\!(\xi,\eta)_{rs_{mn}} =  \widehat{u}(\xi \otimes \eta)_{ij},
\end{equation*}
with
\begin{align*}
i = d_\eta(m-1)+ r \nonumber,  \quad
j  =  d_\eta(n-1) + s \nonumber,
\end{align*}
for all $1\leq m,n \leq d_\xi$ and $1 \leq r,s \leq d_\eta$.

Notice that 
\begin{equation}\label{normHS}
\left\|\widehat{u}(\xi\otimes\eta)\right\|_{\HS}^2 = \sum_{i,j=1}^{d_\xi d_\eta} \left| \widehat{u}(\xi\otimes\eta)_{ij}\right|^2 = \sum_{m,n=1}^{d_\xi} \sum_{r,s=1}^{d_\eta} \left| \doublehat{\,u\,}\!(\xi,\eta)_{mn_{rs}}\right|^2 =: \Big\|\doublehat{\,u\,}\!(\xi,\eta)\Big\|_{\HS}^2,
\end{equation}
for all $[\xi] \in \widehat{G_1}$ and $[\eta] \in \widehat{G_2}$ whenever $u \in L^1(G)$ or $u\in \DG$.

It follows from the \eqref{pseudonorm} and \eqref{normHS} that we have the following characterization of smooth functions and distributions using the method of taking the partial Fourier coefficients twice that was discussed above.
 
\begin{thm}\label{smoo}
	Let $G_1$ and $G_2$ be compact Lie groups, and set $G=G_1\times G_2$ . The following statements are equivalent:
	\begin{enumerate}[(i)]
		\item $f \in C^\infty(G)$;
		\item For every $N>0$, there exists $C_N>0$ such that
	$$
	\|\doublehat{\,f\,}\!(\xi,\eta)\|_{\HS} \leq C_N(\jp{\xi}+\jp{\eta})^{-N},
	$$
	for all $[\xi] \in \widehat{G_1}$ and $[\eta] \in \widehat{G_2}$;
	\item For every $N>0$, there exists $C_N>0$ such that
	$$
	\Big|\doublehat{\,f\,}\!(\xi,\eta)_{mn_{rs}}  \Big| \leq C_N (\jp{\xi}+\jp{\eta})^{-N},
	$$
	for all $[\xi] \in \widehat{G_1}$, $[\eta] \in \widehat{G_2}$, $1 \leq m,n \leq d_\xi$,  and $1 \leq r,s \leq d_\eta$.
\end{enumerate}
\end{thm}

\begin{thm}\label{dist}
	Let $G_1$ and $G_2$ be compact Lie groups, and  set $G=G_1\times G_2$. The following statements are equivalent:
	\begin{enumerate}[(i)]
		\item  $u \in \DG$;
		\item There exist $C, \, N >0$ such that
	$$
	\|\doublehat{\,u\,}\!(\xi,\eta)\|_{\HS} \leq C(\jp{\xi}+\jp{\eta})^{N},
	$$
	 for all $[\xi] \in \widehat{G_1}$ and $[\eta] \in \widehat{G_2}$;
	\item There exist $C, \, N >0$ such that
	$$
	\Big|\doublehat{\,u\,}\!(\xi,\eta)_{mn_{rs}}  \Big| \leq C (\jp{\xi}+\jp{\eta})^{N},
	$$
	for all $[\xi] \in \widehat{G_1}$, $[\eta] \in \widehat{G_2}$, $1 \leq m,n \leq d_\xi$, and $1 \leq r,s \leq d_\eta.$
	\end{enumerate}
\end{thm}

In the next results we will investigate when a sequence of partial Fourier coefficients can define a smooth function or a distribution.

\begin{thm}\label{caracsmooth}
	Let $G_1$ and $G_2$ be compact Lie groups, $G=G_1\times G_2$, and let $\{\widehat{f}(\:\cdot \:, \eta)_{rs} \}$ be a sequence of functions on $G_1$. Define
	$$
	f(x_1,x_2) := \sum_{[\eta]\in \widehat{G_2}} d_\eta \sum_{r,s=1}^{d_\eta} \widehat{f}(x_1, \eta)_{rs} \eta_{sr}(x_2).
	$$ 
	Then $f \in C^\infty(G)$ if and only if $\widehat{f}(\: \cdot\:,\eta)_{rs} \in C^\infty(G_1)$, for all $[\eta]\in\widehat{G_2}$, $1 \leq r,s \leq d_\eta$ and for every $\beta \in \N_0^n$ and $\ell >0$ there exist $C_{\beta \ell} > 0$ such that
	\begin{equation*}\label{carac}
	\bigl|\partial^\beta \widehat{f}(x_1,\eta)_{rs}\bigl| \leq C_{\beta\ell} \langle \eta \rangle ^{-\ell}, \quad \forall x_1\in G_1, \ [\eta]\in \widehat{G}, \ 1\leq r,s\leq d_\eta.
	\end{equation*}
\end{thm}
\begin{proof}
$(\impliedby)$ It is sufficient to consider $N\in \N$ in Theorem \ref{smoo} to conclude that $f\in C^\infty{(G)}$. Notice that
$$
\widehat{\mathcal{L}_{G_1}g}(\xi)_{mn} = \jp{\mathcal{L}_{G_1}g,\overline{\xi_{nm}}} = \jp{g,\mathcal{L}_{G_1}\overline{\xi_{nm}}} = -\nu_{[\xi]} \jp{g,\overline{\xi_{nm}}} = -\nu_{[\xi]} \widehat{g}(\xi)_{mn},
$$
for all $g \in C^{\infty}(G_1)$, $[\xi] \in \widehat{G_1}$, and $1 \leq m,n \leq d_\xi$. In particular, for $N\in \N$ we obtain
\begin{align*}
\nu_{[\xi]}^N |\doublehat{\,f\,}\!(\xi,\eta)_{rs_{mn}}| &= \left|\widehat{\mathcal{L}_{G_1}^N{\widehat{f}}}(\xi,\eta)_{rs_{mn}} \right|\nonumber \\ &= \left|\int_{G_1} \mathcal{L}_{G_1}^N\widehat{f}(x_1,\eta)_{rs} \overline{\xi(x_1)_{nm}} \, dx_1 \right| \nonumber \\
&\leq \int_{G_1}  |\mathcal{L}_{G_1}^N\widehat{f}(x_1,\eta)_{rs}| |{\xi(x_1)_{nm}}|\, dx_1\nonumber \\
&\leq\left(\int_{G_1}  |\mathcal{L}_{G_1}^N\widehat{f}(x_1,\eta)_{rs}|^2\, dx_1\right)^{1/2} \left(\int_G |{\xi(x_1)_{nm}}|^2\, dx_1 \right)^{1/2}\nonumber\\
 &\leq  \frac{1}{\sqrt{d_\xi}} \sum_{|\beta|=2N} \max_{x_1\in G_1} | \partial^\beta\widehat{f}(x_1,\eta)_{rs}|  \\&\leq \sum_{|\beta|=2N} \max_{x_1\in G_1} | \partial^\beta\widehat{f}(x_1,\eta)_{rs}|.
\end{align*}

Notice that there exists $C>0$ such that $\jp{\xi} \leq C \nu_{[\xi]}$ for all non-trivial $[\xi] \in \widehat{G_1}$. Thus for all $\ell=N$ we have 
$$ |\doublehat{\,f\,}\!(\xi,\eta)_{rs_{mn}}| \leq C_{N}\jp{\xi}^{-N}\jp{\eta}^{-N} \leq C_{N}2^N (\jp{\xi} + \jp{\eta})^{-N}.
$$
Therefore $f \in C^{\infty}(G)$.

$(\implies)$ Let $E_2:=(I-\mathcal{L}_{G_2})^{1/2}$. Since $f\in C^\infty(G)$, for all $\beta \in N_0^n$ and $N \in \N_0$ we have $\partial^\beta E_2^Nf \in C^\infty(G)$ and then, by the compactness of $G$, there exists $C_{\beta N}\geq0$ such that
\begin{equation}\label{betaN}
|\partial^\beta E_2^N f(x_1,x_2)| \leq C_{\beta N}, \qquad \forall (x_1,x_2) \in G_1 \times G_2.
\end{equation}
Fix $\eta\in \Rep(G_2)$, $1\leq r,s \leq d_\eta$. We already know that $\widehat{f}(\: \cdot\:,\eta)_{rs} \in C^\infty(G_1)$. Moreover

\begin{align*}
|\jp{\eta}^N\partial^\beta\widehat{f}(x_1,\eta)_{rs}| &= |\partial^\beta\widehat{E_2^Nf}(x_1,\eta)_{rs}| \nonumber \\ &= \left|\partial^\beta\int_{G_2} E_2^Nf(x_1,x_2) \overline{\eta(x_2)_{sr}} \, dx_2\right| \nonumber \\
&\leq \int_{G_2}|\partial^\beta E_2^Nf(x_1,x_2)| |\overline{\eta(x_2)_{sr}}| \, dx_2 \nonumber \\
&\leq \left(\int_{G_2} |\partial^\beta E_2^Nf(x_1,x_2)|^2 \, dx_2\right)^{1/2} \left(\int_{G_2} |\eta(x_2)_{sr}|^2 \, dx_2 \right)^{1/2} \nonumber  \\
&\stackrel{\eqref{betaN}}{\leq} \tfrac{1}{\sqrt{d_\eta}}C_{\beta N} \nonumber.
\end{align*}
Therefore,
$$
|\partial^\beta\widehat{f}(x_1,\eta)_{rs}| \leq C_{\beta N} \jp{\eta}^{-N},
$$
for all $x_1 \in G_1$, $[\eta] \in \widehat{G_2}$, $1\leq r,s\leq d_\eta$.
\end{proof}

\begin{thm}\label{caracdist}
		Let $G_1$ and $G_2$ be compact Lie groups, set $G=G_1\times G_2$, and let $\bigl\{\widehat{u}(\: \cdot \:,\eta)_{rs} \bigr\}$ be a sequence of distributions on $G_1$.  Define \begin{equation*}
	u=\sum_{[\eta]\in \widehat{G}} d_\eta \sum_{r,s=1}^{d_\eta} \widehat{u}(\:\cdot\:,\eta)_{rs}{\eta_{sr}}.
	\end{equation*} 
	Then $u \in \DG$ if and only if there exist $K \in \N$ and $C>0$ such that
	\begin{equation}\label{pk}
	\bigl| \jp{\widehat{u}(\cdot,\eta)_{rs},\varphi}\bigr| \leq C \, p_K(\varphi) \langle \eta \rangle ^{K}, 
	\end{equation} 
	for all $\varphi \in C^{\infty}(G_1)$ and $[\eta]\in \widehat{G}$, where $p_K(\varphi) := \sum\limits_{|\beta| \leq K} \|\partial^\beta \varphi\|_{L^\infty(G_1)}.$
\end{thm}

\begin{proof}
$(\Longleftarrow)$ Take $\varphi=\overline{\xi_{nm}}$, $[\xi]\in \widehat{G_1}$, $1\leq m,n \leq d_\xi$. Let $\beta \in \N_0^n$, $|\beta| \leq K$, with $K$ as in \eqref{pk}. Since the symbol of $\partial^\beta$ at $x_1\in G_1$ and ${\xi \in \Rep(G_1)}$ is given by
$$
\sigma_{\partial^\beta}(x_1,\xi) = \xi(x_1)^*(\partial^\beta \xi)(x_1),
$$
we have
\begin{align*}
|\partial^\beta \xi_{nm}(x_1)| &= \left| \sum_{i=1}^{d_\xi} \xi_{ni}(x) \sigma_{\partial^\beta}(\xi)_{im} \right| \nonumber \\
&\leq \sum_{i=1}^{d_\xi} |\xi_{ni}(x)| |\sigma_{\partial^\beta}(\xi)_{im}| \nonumber \\
&\leq \left(  \sum_{i=1}^{d_\xi} |\xi_{ni}(x)^2| \right)^{1/2} \left(\sum_{i=1}^{d_\xi}|\sigma_{\partial^\beta}(\xi)_{im}|^2  \right)^{1/2} \nonumber 
\end{align*}
Let $M\in\Z$ satisfying $M \geq \frac{\dim G_1}{2}$. By \eqref{estimatexi} we have
$$ 
\left(\sum_{i=1}^{d_\xi} |\xi_{ni}(x)^2|\right)^{1/2}  \leq \left(\sum_{i=1}^{d_\xi}\|\xi_{ni}\|_{L^\infty(G_1)}^2\right)^{1/2}   \leq C_M\sqrt{d_\xi}\jp{\xi}^{M}
$$
and by \eqref{dimension} there exists $C>0$ such that
$$
d_\xi \leq C\jp{\xi}^M.
$$
Moreover, notice that
$$
\left(\sum_{i=1}^{d_\xi}|\sigma_{\partial^\beta}(\xi)_{im}|^2  \right)^{1/2} \leq \|\sigma_{\partial^\beta}(\xi)\|_{\HS} \leq \sqrt{d_\xi}\|\sigma_{\partial^\beta}(\xi)\|_{op} \leq \sqrt{d_\xi}C_0^{|\beta|}  \jp{\xi}^{|\beta|},
$$
where the last inequalities come from \eqref{estimateop} and \eqref{symbol}. Hence
\begin{align*}
|\partial^\beta \xi_{nm}(x_1)| &\leq C \jp{\xi}^M\sqrt{d_\xi} \|\sigma_{\partial^\beta}(\xi)\|_{\HS}\nonumber \\
&\leq C \jp{\xi}^M{d_\xi} \|\sigma_{\partial^\beta}(\xi)\|_{op} \nonumber \\
&\leq CC_0^{|\beta|} \jp{\xi}^{2M+|\beta|} \nonumber.
\end{align*}
Then
$$
p_K(\overline{\xi_{nm}}) = p_K(\xi_{nm}) \leq C \jp{\xi}^{2M+K}.
$$
Hence
\begin{align*}
\bigl| \doublehat{\,u\,}\!(\xi,\eta)_{rs_{mn}}\bigr| = \bigl| \jp{\widehat{u}(\cdot,\eta)_{rs},\overline{\xi_{nm}}}\bigr| &\leq C \, p_K(\overline{\xi_{nm}}) \langle \eta \rangle ^{K} \nonumber \\ 	
&\leq C  \jp{\xi}^{2M+K} \langle \eta \rangle^{2M+K} \nonumber \\
& \leq  C (\langle \xi \rangle + \langle \eta \rangle)^{2(2M+K)}. \nonumber
\end{align*}
Therefore $u \in \mathcal{D}'(G)$.

$ (\Longrightarrow) $ Since $u \in \mathcal{D}'(G)$, then there exist $C>0$ and $K\in\N$ such that
\begin{equation}\label{distr}
\bigl|\doublehat{\,u\,}\!(\xi,\eta)_{rs_{mn}}\bigr| \leq C (\langle\xi\rangle + \langle \eta \rangle)^K, 
\end{equation}
for all $[\xi] \in \widehat{G_1}, [\eta] \in \widehat{G_2}, \  1 \leq r,s\leq d_\eta$, and $1\leq m,n\leq d_\xi$
and 
$$
u=\sum_{[\xi]\in \widehat{G_1}} \sum_{[\eta] \in \widehat{G}} d_\xi d_\eta \sum_{m,n=1}^{d_\xi}\sum_{r,s=1}^{d_\eta} \doublehat{\,u\,}\!(\xi,\eta)_{rs_{mn}} \xi_{nm} \eta_{sr}.
$$

For $\varphi \in C^\infty(G_1)$ we have
{\small 
\begin{align*}
|\jp{\widehat{u}(\cdot,\eta)_{rs},\varphi}| &= |(u,\varphi\times\overline{\eta_{sr}})| \nonumber \\
&= \left|  \sum_{[\xi]\in\widehat{G_1}} \sum_{[\eta] \in \widehat{G}} d_\xi d_\eta \sum_{m,n=1}^{d_\xi}\sum_{k,\ell=1}^{d_\eta} \doublehat{\,u\,}\!(\xi,\eta)_{k\ell_{mn}} \jp{\xi_{nm},\varphi}_{G_1} \jp{\eta_{\ell k}, \overline{\eta_{sr}}}_{G_2}\right|. \nonumber
\end{align*}}

Notice that $\jp{\eta_{\ell k}, \overline{\eta_{sr}}}_{G_2} = \frac{1}{d_\eta}\delta_{\ell s} \delta_{kr}$, since the set $\mathcal{B}$ is orthonormal (see \eqref{ortho}).
 Moreover, $\widehat{\varphi}(\overline{\xi})_{mn} =\jp{\xi_{nm},\varphi}_{G_1}$. So

\begin{align*}
|\jp{\widehat{u}(\cdot,\eta)_{rs},\varphi}| &{=} \left|\sum_{[\xi]\in \widehat{G_1}}d_\xi \sum_{m,n=1}^{d_\xi} \doublehat{\,u\,}\!(\xi,\eta)_{rs_{mn}} \widehat{\varphi}(\overline{\xi})_{mn} \right| \nonumber \\
&\leq \sum_{[\xi]\in \widehat{G_1}}d_\xi \sum_{m,n=1}^{d_\xi} \left|\doublehat{\,u\,}\!(\xi,\eta)_{rs_{mn}} \right|\left|\widehat{\varphi}(\overline{\xi})_{mn} \right| \nonumber \\
&\leq C\sum_{[\xi]\in \widehat{G_1}}d_\xi \sum_{m,n=1}^{d_\xi} (\jp{\xi}+\jp{\eta})^K\left|\widehat{\varphi}(\overline{\xi})_{mn} \right|, \nonumber 
\end{align*}
where the last inequality comes from \eqref{distr}. Notice that for all $K\in \N$ we have $(\jp{\xi}+\jp{\eta})^K \leq 2^K \jp{\xi}^K \jp{\eta}^K$. In addition, we have
$$
\sum_{m,n=1}^{d_\xi}|\widehat{\varphi}({\xi})_{mn}| \leq \left(d_\xi^2\sum_{m,n=1}^{d_\xi}|\widehat{\varphi}({\xi})_{mn}|^2\right)^{1/2} = d_\xi \|\widehat{\varphi}(\xi)\|_{\HS}.
$$
Since $\jp{\xi} = \jp{\overline{\xi}}$ and the summation is over all $\widehat{G_1}$, we have
\begin{align*}
 \sum_{[\xi]\in \widehat{G_1}} d_\xi \langle \xi \rangle^{K}\sum_{m,n=1}^{d_\xi}|\widehat{\varphi}(\overline{\xi})_{mn}| &=  \sum_{[\xi]\in \widehat{G_1}} d_{\overline{\xi}} \langle \overline{\xi} \rangle^{K}\sum_{m,n=1}^{d_\xi}|\widehat{\varphi}({\xi})_{mn}| \\ &=  \sum_{[\xi]\in \widehat{G_1}} d_\xi \langle \xi \rangle^{K}\sum_{m,n=1}^{d_\xi}|\widehat{\varphi}({\xi})_{mn}|.
 \end{align*}
 Thus
\begin{align*}
|\jp{\widehat{u}(\cdot,\eta)_{rs},\varphi}| 
& \leq  C\jp{\eta}^K \sum_{[\xi]\in \widehat{G_1}} d_\xi \langle \xi \rangle^{K}\sum_{m,n=1}^{d_\xi}|\widehat{\varphi}(\overline{\xi})_{mn}|\nonumber \\ 
& \leq  C\jp{\eta}^K \sum_{[\xi]\in \widehat{G_1}} d_\xi^2 \langle \xi \rangle^{K}\|\widehat{\varphi}(\xi)\|_{\HS}. \nonumber
\end{align*}
The series $\sum\limits_{[\xi]\in \widehat{G_1}} d_\xi^2 \langle \xi \rangle^{-2t}$ 
converges if and only if $t > \frac{\dim G_1}{2}$ (Lemma 3.1 of \cite{DR14}), which implies that there exists $C>0$ such that $d_\xi \leq C \jp{\xi}^{\dim G_1}$, for all $[\xi] \in \widehat{G_1}$. Hence,

\begin{align*}
|\jp{\widehat{u}(\cdot,\eta)_{rs},\varphi}|&= C \jp{\eta}^K \sum_{[\xi]\in \widehat{G_1}} \left(d_\xi \langle \xi \rangle^{-\dim G_1}\right) \left(d_\xi\jp{\xi}^{K+\dim G_1}\|\widehat{\varphi}(\xi)\|_{\HS}\right) \nonumber \\ &\leq  C \jp{\eta}^K \left(\sum_{[\xi]\in \widehat{G_1}} d_\xi^2 \langle \xi \rangle^{-2\dim G_1}\right)^{1/2} 	 \left(\sum_{[\xi]\in \widehat{G_1}} d_\xi^2\jp{\xi}^{2(K+\dim G_1)}\|\widehat{\varphi}(\xi)\|_{\HS}^2\right)^{1/2} \nonumber \\ 
&\leq C \jp{\eta}^K\left(\sum_{[\xi]\in \widehat{G_1}} d_\xi \jp{\xi}^{2(K+2\dim G_1)}\|\widehat{\varphi}(\xi)\|_{\HS}^2\right)^{1/2}. \nonumber 
\end{align*}

Let $L \in \N_0$ such that $K+2\dim G_1 \leq 2L$. So
\begin{align*}
|\jp{\widehat{u}(\cdot,\eta)_{rs},\varphi}|&\leq C \jp{\eta}^K\left(\sum_{[\xi]\in \widehat{G_1}} d_\xi \jp{\xi}^{4L}\|\widehat{\varphi}(\xi)\|_{\HS}^2\right)^{1/2} 
\nonumber \\&= C \jp{\eta}^K\left(\sum_{[\xi]\in \widehat{G_1}} d_\xi \|\widehat{E^{2L}_1\varphi}(\xi)\|_{\HS}^2\right)^{1/2} \nonumber \\ 
& =   C\jp{\eta}^K \|E^{2L}_1\varphi\|_{L^2(G_1)}\nonumber \\
&\leq  C \|E^{2L}_1\varphi\|_{L^2(G_1)} \jp{\eta}^{2L}\nonumber, 
\end{align*}
where $E_1 = (I-\mathcal{L}_{G_1})^{1/2}$, and the last equality comes from the Plancherel formula \eqref{plancherel}. Notice that
$$
\|E^{2L}_1\varphi\|_{L^2(G_1)} \leq \|E^{2L}_1\varphi\|_{L^\infty(G_1)} = \|(I-\mathcal{L}_{G_1})^L\varphi\|_{L^\infty(G_1)} \leq Cp_{2L}(\phi).
$$
Therefore,
$$
|\jp{\widehat{u}(\cdot,\eta)_{rs},\varphi}| \leq Cp_{2L}(\phi)\jp{\eta}^{2L}.
$$
\end{proof}

\section{Example: Global solvability of a vector field on $\mathbb{T}^1 \times \St$}
The 3-sphere $\St$ is a Lie group with respect to the quaternionic 
product of $\mathbb{R}^{4}$, and it is globally diffeomorphic and isomorphic to the
group $\mbox{SU}(2)$ of unitary $2\times 2$ matrices of determinant one,
with the usual matrix product.

Let $a(t)\in \mathbb{R}$ for $t\in\mathbb{T}^1$. Let $X$ be a normalized left-invariant vector field on $\St$.
We consider the operator
\begin{equation}\label{operator}
L=\partial_{t}+a(t)X.
\end{equation}
We are interested in solvability properties for the vector field $L$ on $\mathbb{T}^1\times\St$.
Let $\Sth$ be the unitary dual of $\St$. It consists of the
equivalence classes $[\textsf{t}^\ell]$ of the continuous irreducible unitary representations
$\textsf{t}^\ell:\St\to \C^{(2\ell+1)\times (2\ell+1)}$, $\ell\in\frac12\N_{0}$,
of matrix-valued functions satisfying
$\textsf{t}^\ell(xy)=\textsf{t}^\ell(x)\textsf{t}^\ell(y)$ and $\textsf{t}^\ell(x)^{*}=\textsf{t}^\ell(x)^{-1}$ for all
$x,y\in\St$.
We will use the standard convention of enumerating the matrix elements
$\textsf{t}^\ell_{mn}$ of $\textsf{t}^\ell$ using indices $m,n$ ranging between
$-\ell$ to $\ell$ with step one, i.e. we have $-\ell\leq m,n\leq\ell$
with $\ell-m, \ell-n\in\N_{0}.$

For a function
$f\in C^{\infty}(\St)$ we can define its Fourier coefficient at $\ell\in\frac12\N_{0}$ by
$$
\widehat{f}(\ell):=\int_{\St} f(x) \textsf{t}^{\ell}(x)^{*}dx\in\Ctl,
$$
where the integral is (always) taken with respect to the Haar measure on $\St$,
and with a natural extension to distributions.
The Fourier series becomes
$$
f(x)=\sum_{\ell\in\frac12\N_{0}} (2\ell+1) \mbox{Tr}\p{\textsf{t}^\ell(x)\widehat{f}(\ell)},
$$
with the Plancherel's identity taking the form
\begin{equation}\label{EQ:Plancherel}
\|f\|_{L^{2}(\St)}=\p{\sum_{\ell\in\frac12\N_{0}} (2\ell+1) 
	\|\widehat{f}(\ell)\|_{\HS}^{2}}^{1/2}=:
\|\widehat{f}\|_{\ell^{2}(\St)},
\end{equation}
which we take as the definition of the norm on the Hilbert space
$\ell^{2}(\Sth)$, and where 
$ \|\widehat{f}(\ell)\|_{\HS}^{2}=\emph{\Tr}(\widehat{f}(\ell)\widehat{f}(\ell^{*}))$ is the
Hilbert--Schmidt norm of the matrix $\widehat{f}(\ell)$.

Smooth functions and distributions on $\St$ can be characterized in terms of
their Fourier coefficients. Thus, we have 
$$
f\in C^{\infty}(\St)\Longleftrightarrow
\forall N \;\exists C_{N} \textrm{ such that }
\|\widehat{f}(\ell)\|_{\HS}\leq C_{N} (1+\ell)^{-N}, \forall \ell\in\frac12\N_{0}.
$$ 
Also, for distributions, we have
$$
u\in  \mathcal{D}'(\St)
\Longleftrightarrow
\exists M \;\exists C \textrm{ such that }
\|\widehat{u}(\ell)\|_{\HS}\leq C(1+\ell)^{M}, \forall \ell\in\frac12\N_{0}.
$$

Given an operator $T:C^{\infty}(\St)\to C^{\infty}(\St)$
(or even $T:C^{\infty}(\St)\to \mathcal{D}'(\St)$), we define its matrix symbol by
$$\sigma_{T}(x,\ell):=\textsf{t}^\ell(x)^{*} (T \textsf{t}^\ell)(x)\in \Ctl,$$ where
$T \textsf{t}^\ell$ means that we apply $T$ to the matrix components of $\textsf{t}^\ell(x)$.
In this case we can prove that 
\begin{equation}\label{EQ:T-op}
Tf(x)=\sum_{\ell\in\frac12\N_{0}} (2\ell+1) \mbox{Tr}\left({\textsf{t}^\ell(x)\sigma_{T}(x,\ell)\widehat{f}(\ell)}\right).
\end{equation}
The correspondence between operators and symbols is one-to-one, and
we will write $T_{\sigma}$ for the operator given by 
\eqref{EQ:T-op} corresponding to the symbol $\sigma(x,\ell)$.
The quantization \eqref{EQ:T-op} has been extensively studied in
\cite{livropseudo,RT13}, to which we
refer for its properties and for the corresponding symbolic calculus.

Using rotation on $\St$, without loss of generality we may assume that the vector
field $X$ has the symbol 
$$
\sigma_{X}(\ell)_{mn}=im\delta_{mn},
\;-\ell\leq m,n\leq \ell,
$$ 
with $\delta_{mn}$ standing for the Kronecker's delta.
Consequently, taking the Fourier transform of
$Lu=\partial_{t}u+a(t)Xu$ with respect to $x$, we get 
$$\widehat{Lu}(t,\ell)= \partial_{t} \widehat{u}(t,\ell)+ia(t)\widehat{Xu}(t,\ell).$$
Writing this in terms of matrix coefficients, we obtain
\begin{equation}\label{equation}
\widehat{Lu}(t,\ell)_{mn}= \partial_{t} \widehat{u}(t,\ell)_{mn}+ia(t) m
\widehat{u}(t,\ell)_{mn},\qquad
-\ell\leq m,n\leq \ell.
\end{equation}
Consequently, the equation $Lu=f$, where $f \in C^\infty(\mathbb{T}^1\times \St)$ is reduced to
\begin{equation}\label{EQ:LufFC}
\partial_{t} \widehat{u}(t,\ell)_{mn}+ia(t) m\
\widehat{u}(t,\ell)_{mn}=\widehat{f}(t,\ell)_{mn},\quad
\ell\in\frac12\N_{0},\;
-\ell\leq m,n\leq \ell.
\end{equation}

We say that an operator $P$ is global hypoelliptic if the conditions $Pu \in C^\infty(\mathbb{T}^1\times \St)$ and $u \in \mathcal{D}'(\mathbb{T}^1\times \St)$ imply $u \in C^\infty(\mathbb{T}^1\times \St)$. 

We point out that the operator $L$ defined in \eqref{operator} is not globally hypoelliptic for any $a\in C^{\infty}(\mathbb{T}^1)$. Indeed, define $u$ by its partial Fourier coefficients:
\begin{equation}\label{counter}
\widehat{u}(t,\ell)_{mn} = \left\{\begin{array}{ll} 1, &\mbox{if } \ell \in \N_0 \mbox{ and } m=n=0; \\
0, & \mbox{otherwise}.
\end{array} \right.
\end{equation}
Notice that by Theorem \ref{caracsmooth} we have $u \notin C^\infty(\TS)$ and by Theorem \ref{caracdist} we have $u \in \mathcal{D}'(\TS)$. Moreover, the function defined by \eqref{counter} is a homogeneous solution of \eqref{EQ:LufFC} for all $\ell \in \frac{1}{2}\N_0$, $-\ell \leq m,n \leq \ell$. Thus $Lu=0 \in C^\infty(\TS)$ which implies that $L$ is not globally hypoelliptic.

Let us turn our attention now to study the solvability of the operator $L$. Denote by $$a_0= \frac{1}{2\pi} \int_{0}^{2\pi} a(t)\, dt$$ and define
\begin{equation}\label{primitiva} A(t)=\int_{0}^{t} a(s) \, ds - a_0t.\end{equation}

Put
$$
v(t,\ell)_{mn}:= e^{imA(t)}\widehat{u}(t,\ell)_{mn}.
$$
In this way $v(t,\ell)_{mn}$ satisfies the equation
\begin{equation}\label{equationconst}
\partial_tv(t,\ell)_{mn}+ima_0v(t,\ell)_{mn}=g(t,\ell)_{mn}:= e^{imA(t)}\widehat{f}(t,\ell)_{mn}. 
\end{equation}
\begin{lemma}\label{solutionconstant}
	Let $\lambda \in \C$ and consider the equation
	\begin{equation}\label{equationconstant}
	\frac{d}{dt}u(t)+\lambda u(t)=f(t),
	\end{equation}
	where $f \in C^\infty(\mathbb{T}^1)$.
	
	If $\lambda \notin i\Z$ then the equation \eqref{equationconstant} has a unique solution that can be expressed by
	\begin{equation}\label{solutionminus}
	u(t)=\frac{1}{1-e^{-2\pi\lambda}} \int_0^{2\pi} e^{-\lambda s} f(t-s)\, ds,
	\end{equation}
	or equivalently,
	\begin{equation}\label{solutionplus}
		u(t)=\frac{1}{e^{2\pi\lambda}-1} \int_0^{2\pi} e^{\lambda r} f(t+r)\, dr.
	\end{equation}
	
	If $\lambda \in i\Z$ and $\int_0^{2\pi} e^{\lambda s}f(s)\, ds=0$ then we have that
	\begin{equation}\label{solutionzero}
	u(t)=e^{-\lambda t}\int_0^t e^{\lambda s}f(s) \, ds
	\end{equation}
	is a solution of the equation \eqref{equationconstant}.
\end{lemma}
To prove this lemma observe that $E=(1-e^{-2\pi\lambda})^{-1}e^{\lambda t}$ is the fundamental solution of the operator $\frac{d}{dt}+\lambda$, when $\lambda \notin i\Z$. The equivalence between \eqref{solutionminus} and \eqref{solutionplus} follows from the change of variable $s \mapsto -r+2\pi$. We point out that the solution \eqref{solutionzero} is not unique.

By Lemma \ref{solutionconstant}, if $ma_0 \notin \Z$  the equation \eqref{equationconst} has a unique solution that can be written as
$$
v(t,\ell)_{mn} = \frac{1}{1-e^{-2\pi ima_0}} \int_0^{2\pi} e^{-ima_0 s} g(t-s,\ell)_{mn} \, ds
$$
and then
\begin{equation}\label{solution1}
\widehat{u}(t,\ell)_{mn} = \frac{1}{1-e^{-2\pi ima_0}} \int_0^{2\pi} e^{-imH(t,s)} \widehat{f}(t-s,\ell)_{mn} \, ds,
\end{equation}
where $H(t,s) := \int_{t-s}^{t} a(\theta) \, d\theta.$

If $ma_0 \in \Z$, a solution for the equation \eqref{equationconst} can be expressed as
$$
v(t,\ell)_{mn} = e^{-ima_0t}\int_0^t e^{ima_0 s} g(s,\ell)_{mn} \, ds
$$
and then
\begin{equation}\label{solution2}
\widehat{u}(t,\ell)_{mn} = e^{-imH(t,t)} \int_0^t e^{imH(s,s)} \widehat{f}(s,\ell)_{mn} \, ds. 
\end{equation}
Here, one can see a condition on $f$ for the existence of a solution $\widehat{u}(t,\ell)_{mn}$ on $\mathbb{T}^1$. If $m a_0 \in \Z$, then
\begin{equation}\label{compatibility}
\int_0^{2\pi} e^{imH(t,t)} \widehat{f}(t,\ell)_{mn} \, dt=0.
\end{equation}
This condition appears in Lemma \ref{solutionconstant} in order to guarantee that the solution is well-defined in $\mathbb{T}^1$.

So, if $f \in C^\infty(\TS)$ does not satisfy the condition above, the equation \eqref{EQ:LufFC} has no solution on $\mathbb{T}^1$. We will denote by $\mathcal{K}$ the set of smooth functions on $\TS$ that satisfy the condition \eqref{compatibility}.

\begin{defi}\label{defsolv}
	We say that the operator $L$ is globally $C^\infty$-solvable if for any $f \in \mathcal{K}$, there exists $u \in C^\infty(\TS)$ such that $Lu=f$.
\end{defi}

\begin{rem}
	In the literature it is common to define the global $C^\infty$-solvability for $f$ in the set
	$$
	\{h\in C^\infty(\TS); \jp{w,h}=0; \ \forall w \in \KER\, ^t\!L \},
	$$
	(e.g. \cite{Pet98}, \cite{Pet11}). For the operator $L$ that we are studying in this work, this set coincides with the set $\mathcal{K}$ defined previously.
\end{rem}
Let us investigate if the operator $L$ defined in \eqref{operator} is globally $C^\infty$-solvable. First, let us prove a technical result about derivatives of the exponential function.
\begin{lemma}\label{lemma51}
	For any $\alpha \in \N_0$, there exists $C_\alpha$ such that
	$$
	|\partial^\alpha_t e^{imH(t,t)}| \leq C_\alpha |m|^{\alpha},
	$$
	for all $t \in \mathbb{T}^1$.
\end{lemma}
\begin{proof}
	By Fa\`a di Bruno's Formula, we have
	$$
	\partial_t^\alpha e^{imH(t,t)} = \sum_{\gamma \in \Delta(\alpha)} \frac{\alpha! }{\gamma!} (im)^{|\gamma|} e^{imH(t,t)} \prod_{j=1}^{\alpha} \left(\frac{\partial_t^j H(t,t)}{j!}\right)^{\gamma_j},
	$$
	where $\Delta(\alpha) = \bigg\{\gamma \in \N_0^{\alpha}; \ \sum\limits_{j=1}^\alpha j\gamma_j= \alpha \bigg\}$.
	
	Notice that $H(t,t)=\int_0^t a(\theta) \, d\theta$, so $\partial^j_tH(t,t) = \partial_t^{j-1} a(t)$, for all $j\geq 1$. Hence
	$$
 	|\partial^\alpha_t e^{imH(t,t)}| \leq \sum_{\gamma \in \Delta(\alpha)} \frac{\alpha! }{\gamma!} |m|^{|\gamma|} \prod_{j=1}^{\gamma} \left(\frac{|\partial_t^j H(t,t)|}{j!}\right)^{\gamma_j} \leq C_\alpha |m|^\alpha.
 	$$
\end{proof}

Assume that $f \in C^\infty(\mathbb{T}^1\times \St)$. By Theorem \ref{caracsmooth}, for all $\beta \in \N_0$ and $\N >0$ there exists $C_{\beta N}>0$ such that
\begin{equation}\label{f}
|\partial_t^\beta \widehat{f}(t,\ell)_{mn}| \leq C_{\beta N}(1+\ell)^{-N},
\end{equation}
for all $t \in \mathbb{T}^1$, $\ell \in \frac{1}{2}\N_0$, and $-\ell \leq m,n \leq \ell$.

Let us determine when $u$ defined by the partial Fourier coefficients \eqref{solution1} and \eqref{solution2} belongs to $C^\infty(\mathbb{T}^1\times\St)$.

Let $\alpha \in \N_0$ and $N>0$. If $ma_0 \in \Z$, we have
\begin{equation*}\label{solution3}
\partial_t^\alpha \widehat{u}(t,\ell)_{mn} = \sum_{\beta \leq \alpha} \binom{\alpha}{\beta} \partial_t^{\alpha - \beta}\{e^{-imH(t,t)}\}  \partial_t^\beta\left\{ \int_0^t e^{imH(s,s)} \widehat{f}(s,\ell)_{mn} \, ds\right\}.
\end{equation*}

Notice that for $\beta \geq 1$ we have 
{
\small
\begin{align*}
 \partial_t^\beta\left\{ \int_0^t e^{imH(s,s)} \widehat{f}(s,\ell)_{mn} \, ds\right\} &= \partial_t^{\beta-1} \left\{ e^{imH(t,t)} \widehat{f}(t,\ell)_{mn} \right\}  \\ &= \sum_{\gamma \leq \beta-1} \binom{\beta-1}{\gamma} \partial^{(\beta-1)-\gamma}_t e^{imH(t,t)} \partial_t^\gamma \widehat{f}(t,\ell)_{mn}.
\end{align*}}
For each $\gamma$, by \eqref{f} there exists $C_{\gamma \alpha N}>0$ such that $$|\partial_t^\gamma \widehat{f}(t,\ell)_{mn}| \leq C_{\gamma \alpha N}(1+\ell)^{-(N+\alpha)}.$$
By Lemma \ref{lemma51}, since $|m| \leq \ell$,  we have
$$
| \partial^{(\beta-1)-\gamma}_t e^{imH(t,t)}| \leq C_\beta |m|^{(\beta-1)-\gamma} \leq C_\beta (1+\ell)^\beta.
$$
So, there exists $C_{\alpha\beta N}>0$ such that 
\begin{equation}\label{estimate2}
\left| \partial_t^\beta\left\{ \int_0^t e^{imH(s,s)} \widehat{f}(s,\ell)_{mn} \, ds\right\} \right| \leq C_{\alpha\beta N} (1+\ell)^{-(N+\alpha)+\beta}.
\end{equation}
Again by Lemma \ref{lemma51} we have
\begin{equation}\label{estimate1}
|\partial_t^{\alpha - \beta}\{e^{-imH(t,t)}\}| \leq C_{\alpha \beta}(1+\ell)^{\alpha-\beta}
\end{equation}
Therefore, from estimates \eqref{estimate2} and \eqref{estimate1}, we obtain $C_{\alpha N} >0$ such that
\begin{equation*}
|\partial_t^\alpha \widehat{u}(t,\ell)_{mn}| \leq C_{\alpha N} (1+\ell)^{-N}.
\end{equation*}

Let us assume now that $ma_0 \not \in \Z$. With a slight change in the proof of Lemma \ref{lemma51} we obtain the same estimate
$$
|\partial_t^\alpha e^{imH(t,s)}| \leq C_\alpha |m|^{\alpha},
$$
for all $t,s \in \mathbb{T}^1$.

By the same arguments of the previous case we obtain for all $\alpha \in \N_0$ and $N>0$
$$
|\partial_t^\alpha\widehat{u}(t,\ell)_{mn}| \leq |1-e^{-2\pi im a_0}|^{-1} C_{\alpha N}(1+\ell)^{-N}.
$$

Now, we need to estimate $|1-e^{-2\pi im a_0}|^{-1}$ in order to apply Theorem \ref{caracsmooth} to conclude that $u \in C^\infty(\mathbb{T}^1\times \St)$. 
In the case where $m a_0 \notin \Z$, there exist $C,M>0$ such that 
$$
|1-e^{-2\pi im a_0}| \geq C|m|^{-M},
$$
if and only if $a_0 \in \Q$ or $a_0$ is an irrational non-Liouville number (see \cite{BDGK15}, Lemma 3.4)

This way, for either $a_0 \in \Q$ or $a_0$ an irrational non-Liouville number, we obtain constants $C,M>0$ such that 

$$
|1-e^{-2\pi im a_0}|^{-1} \leq C|m|^{M} \leq C\ell^M,
$$
when $ma_0 \notin\Z$. So, adjusting the constants if necessary, for all $\alpha \in \N_0$ and $N>0$, there exists $C_\alpha N>0$ such that
$$
|\partial^\alpha_t \widehat{u}(t,\ell)_{mn}| \leq C_{\alpha N} (1+\ell)^{-N},
$$
for all $t \in \mathbb{T}^1$, $\ell \in \frac{1}{2}\N_0$, and $-\ell \leq m,n \leq \ell$. Therefore, by Theorem \ref{caracsmooth} we have $u \in C^\infty(\mathbb{T}^1\times \St)$ and by the uniqueness of Fourier coefficients, we conclude that $Lu=f$.


We have proved the following proposition:
\begin{prop}\label{sufficondition}
	The operator $L=\partial_t + a(t)X$ is globally $C^\infty$-solvable if $a_0 \in \Q$ or $a_0$ is an irrational non-Liouville number, where $a_0 = \frac{1}{2\pi}\int_0^{2\pi} a(s) \, ds$.
\end{prop}

\subsection{Normal form}
\mbox{}

Notice that the global solvability of the operator $L$ is strictly related with the number $a_0$. This is not a coincidence because we can conjugate the operator $L$ to the operator $L_{a_0} = \partial_t + a_0 X$. We say that the operator $L_{a_0}$ is the normal form of $L$. In this section we will construct this conjugation and show that the sufficient conditions on Proposition \ref{sufficondition} for the global $C^\infty$--solvability of $L$ are actually necessary conditions.

Define
$$\Psi_a u(t,x) := \sum_{\ell \in \frac{1}{2} \N_0} (2\ell +1) \sum_{m,n=1}^{2\ell+1} e^{imA(t)}\widehat{u}(t,\ell)_{mn}\textsf{t}^\ell(x)_{nm}.
$$
\begin{prop} $\Psi_a$ is an automorphism in $\mathcal{D}'(\mathbb{T}^1\times \St)$ and in $C^\infty(\mathbb{T}^1\times \St)$.
\end{prop}
\begin{proof}
It is easy to see that $\Psi_a$ is linear and has inverse $\Psi_{-a}$, therefore we only need to prove that $\Psi_a (C^\infty(\TS)) = C^{\infty}(\TS)$ and $\Psi_a (\mathcal{D}'(\TS)) = \mathcal{D}'(\TS)$.

Let $\beta \in \N_0$ and $u\in C^\infty(\TS)$. We will use  Theorem \ref{caracsmooth} to show that $\Psi_a u \in C^\infty(\TS)$. Notice that $\widehat{\Psi_a u}(t,\ell)_{mn} = e^{i m A(t)}\widehat{u}(t,\ell)_{mn}$ for all $\ell \in \frac{1}{2}\N_0$, $\ell \leq m,n \leq \ell $ and $t \in \mathbb{T}^1$. In this way
\begin{align*}
|\partial^\beta \widehat{\Psi_a u}(t,\ell)_{mn}| &= | \partial^\beta (e^{imA(t)}\widehat{u}(t,\ell)_{mn})| \nonumber \\
&= \left|\sum_{\gamma \leq \beta } \binom{\beta}{\gamma} \partial^{\beta-\gamma}e^{imA(t)}\partial^{\gamma}\widehat{u}(t,\eta)_{mn} \right|\nonumber \\
&\leq \sum_{\gamma \leq \beta } \binom{\beta}{\gamma} |\partial^{\beta-\gamma}e^{imA(t)}||\partial^{\gamma}\widehat{u}(t,\eta)_{mn}| \nonumber \\
&\leq \sum_{\gamma\leq \beta }C_{\beta} |m|^\beta  \left|\partial^{\beta} \widehat{u}(t,\ell)_{mn}\right| \nonumber,
\end{align*}
where the last inequality comes from an adaptation of Lemma \ref{lemma51}. Since $u \in C^\infty(\TS)$ and $|m| \leq \ell$, again by Theorem \ref{caracsmooth}, it is easy to see that given $N>0$, there exists $C_{\beta N}$ such that
$$
|\partial^\beta \widehat{\Psi_a u}(t,\ell)_{mn}| \leq C_{\beta N} (1+\ell)^{-N}.
$$
Therefore $\Psi_a u \in C^\infty(\TS)$. The distribution case is analogous.

\end{proof}

\begin{prop}\label{normal}
	Let
	$
	L_{a_0}=\partial_{t}+a_0 X.
	$
	Then
$$\Psi_a\circ L = L_{a_0}\circ \Psi_a.$$

\end{prop}
\begin{proof}
 We will show that for every $u \in C^\infty(\mathbb{T}^1\times \St)$ we have $$\widehat{\Psi_a (Lu)}(t,\ell)_{mn}=\widehat{L_{a_0}(\Psi_a u)}(t,\ell)_{mn},$$ for all $t \in \mathbb{T}^1$, $\ell \in \frac{1}{2}\N_0$ and $-\ell \leq m,n \leq \ell$.
Indeed, we have
\begin{align*}
	\widehat{L_{a_0}(\Psi_a u)}(t,\ell)_{mn} & =  \partial_t \widehat{\Psi_a u}(t,\ell)_{mn}+ia_0m\widehat{\Psi_a u}(t,\ell)_{mn} \\
	&= \partial_t \left\{e^{imA(t)} \widehat{u}(t,\ell)_{mn} \right\} + ia_0 me^{imA(t)}\widehat{u}(t,\ell)_{mn} \\
	&= im(a(t)-a_0)e^{imA(t)} \widehat{u}(t,\ell)_{mn} + e^{imA(t)} \partial_t\widehat{u}(t,\ell)_{mn}+ia_0me^{imA(t)} \widehat{u}(t,\ell)_{mn} \\
	&=e^{imA(t)} (\partial_t \widehat{u}(t,\ell)_{mn}+ia(t)m\widehat{u}(t,\ell)_{mn})\\
	&= e^{imA(t)} \widehat{Lu}(t,\ell)_{mn}\\
	&=\widehat{\Psi_a(Lu)}(t,\ell)_{mn}.
\end{align*}
Therefore $\Psi_a\circ L = L_{a_0}\circ \Psi_a$. The same holds when  it is considered $u \in \mathcal{D}'(\mathbb{T}^1 \times \St)$.
\end{proof}

Let $a_0$ be an irrational Liouville number. The Fourier coefficients of $L_{a_0}u$, for $u\in \mathcal{D}'(\mathbb{T}^1\times \St)$ are expressed by
\begin{equation}\label{solution}
\doublehat{L_{a_0}u}(\tau,\ell)_{mn} = i(\tau + a_0m)\doublehat{u}(\tau,\ell)_{mn},
\end{equation}
with $\tau \in \Z$, $\ell \in \frac{1}{2}\N_0$, $-\ell \leq m,n \leq \ell$.

Since $a_0$ is an irrational Liouville number, for every $N\in \N$, there exists $\tau_M \in \Z$ and $\ell_M \in \N$ such that
\begin{equation}\label{liouville}
0 < |\tau_M+a_0\ell_M| \leq (|\tau_M|+|\ell_M|)^{-M}.
\end{equation}
Define
$$
\doublehat{\,f\,}(\tau,\ell)_{mn} =  \left\{
\begin{array}{ll}
\tau_M+a_0\ell_M, & \mbox{if $(\tau,\ell)=(\tau_M,\ell_M)$ for some $M\in \N$ and $m=\ell_M$} , \\
0, & \mbox{otherwise.}
\end{array}
\right.
$$

It is easy to see that \eqref{liouville} and Theorem \ref{smoo} imply that $f \in \mathcal{K}_{a_0}$ (see Definition \ref{defsolv}). If $L_{a_0}u=f$ for some $u\in C^\infty(\mathbb{T}^1\times \St)$, the expression \eqref{solution} gives us 
$$
\left|\doublehat{\, u\, }(\tau_M,\ell_M)_{\ell_M\ell_M}\right|=1,
$$
which contradicts the fact that $u$ is a smooth function. Therefore $L_{a_0}$ is not globally $C^\infty$--solvable when $a_0$ is an irrational Liouville number. Notice that $u \in \mathcal{D}'(\mathbb{T}^1\times \St)$ (see Theorem \ref{dist}). Since $f \in \mathcal{K}_{a_0}$, when $ma_0\in \Z$ we obtain
\begin{align*}
0 = \int_0^{2\pi} e^{ima_0t}\widehat{f}(t,\ell)_{mn} \, dt &= \int_0^{2\pi}e^{imH(t,t)}e^{-imA(t)}\widehat{f}(t,\ell)_{mn} \, dt \\
&= \int_0^{2\pi}e^{imH(t,t)}\widehat{\Psi_{-a}f}(t,\ell)_{mn} \, dt.
\end{align*}
So $\Psi_{-a}f \in \mathcal{K}$. Assume that there exists some $u \in C^{\infty}(\mathbb{T}^1\times \St)$ such that $Lu=\Psi_{-a}f$. By Proposition \ref{normal} we obtain
$$
f = \Psi_a Lu = L_{a_0} \Psi_a u.
$$
By what was discussed previously and the fact that $\Psi_a$ is an automorphism of $C^\infty(\mathbb{T}^1\times \St)$ and $\mathcal{D}'(\mathbb{T}^1\times \St)$, we conclude that $u\in \mathcal{D}'(\mathbb{T}^1\times \St)\setminus C^\infty(\mathbb{T}^1\times \St)$. We have proved the following proposition:
\begin{prop}
	The operator $L=\partial_t + a(t)X$ is globally $C^\infty$-solvable if and only if $a_0 \in \Q$ or $a_0$ is an irrational non-Liouville number, where $a_0 = \frac{1}{2\pi} \int_0^{2\pi} a(s) \, ds$.
\end{prop}
\section*{Acknowledgments}

This study was financed in part by the Coordenação de Aperfeiçoamento de Pessoal de Nível Superior - Brasil (CAPES) - Finance Code 001. The last
author was also supported by the FWO Odysseus grant, by the Leverhulme Grant RPG-2017-151, and by EPSRC Grant EP/R003025/1.
\bibliographystyle{abbrv}
\bibliography{biblio}

\end{document}